\documentclass[preprint,12pt]{elsarticle}
\biboptions{sort&compress}
\usepackage[T1]{fontenc}
\usepackage[utf8]{inputenc}
\usepackage{lmodern}
\usepackage{amsmath,amssymb,mathtools,amsthm}
\newtheorem{theorem}{Theorem}

\newtheorem{lemma}{Lemma}
\newtheorem{proposition}{Proposition}
\usepackage{graphicx,booktabs,array,adjustbox,multirow}
\usepackage{etoolbox}
\usepackage[ruled,vlined]{algorithm2e}
\usepackage[section]{placeins}
\usepackage{setspace,needspace}
\usepackage[font=small,labelfont=bf]{caption}
\AtBeginEnvironment{figure}{\setstretch{1}}
\AtBeginEnvironment{table}{\setstretch{1}}
\AtBeginEnvironment{algorithm}{\setstretch{1}}
\usepackage[hidelinks]{hyperref}
\newcommand{\eps}{\varepsilon}
\newcommand{\dd}{\,\mathrm{d}}
\journal{Journal of Computational Physics}
\begin{document}
\begin{frontmatter}
\title{Constrained weak identification of Allen--Cahn free energies with surface-tension calibration}
\author[kentech]{Soobeen Jung}
\ead{sbeenjung@kentech.ac.kr}
\author[kentech]{Hyunju Kim\corref{cor1}}
\ead{hjkim@kentech.ac.kr}
\address[kentech]{Department of Energy Engineering, Korea Institute of Energy Technology, Naju
58330, Republic of Korea}
\cortext[cor1]{Corresponding author. Tel.: +82-61-320-9240.}
\begin{abstract}
Phase-field trajectories governed by the Allen--Cahn equation determine the effective force $G=F'/\eps^2$, but not the free-energy potential $F$ and interfacial scale $\eps$ separately. We introduce surface-tension-calibrated Allen--Cahn identification (STAC), a two-step method that resolves this ambiguity using an independently supplied surface tension. First, space--time weak moments produce a linear inverse problem without differentiating the measured field. A Bernstein representation with positive coefficients enforces two stable pure phases and a single central energy maximum. Second, the surface tension fixes the remaining scale algebraically. We prove scale equivalence of the forward model, uniqueness of the calibrated pair within this class, and finite perturbation bounds that remain valid when the potential vanishes at the pure phases. Experiments with resolved synthetic phase fields separate the effects of the structural constraint, ridge regularization, and constitutive approximation error. Across 80 shared noisy records, unconstrained degree-five regression produces six potentials with an incorrect well structure. Positivity combined with ridge regularization eliminates these violations and reduces the mean force error from 8.70\% to 5.66\%. Lower-degree models are nevertheless more accurate in some tests, and an admissible law outside the positive Bernstein family shows how the constraint can introduce approximation bias. STAC thus recovers a structurally admissible, physically scaled potential while making explicit which information comes from the observed dynamics and which comes from the independent surface-tension datum.

\end{abstract}
\end{frontmatter}
\raggedbottom
\section{Introduction}
\label{sec:introduction}
Time-resolved phase-field images contain information about the constitutive laws that drive interface motion. In the Allen--Cahn equation, the bulk potential controls local phase stability, while the gradient term penalizes spatial variation \cite{Cahn1958,AllenCahn1979}. Their balance also connects diffuse-interface evolution to motion by mean curvature in the appropriate small-interface-width limit \cite{Chen1992,EvansSonerSouganidis1992,Modica1987}. We consider the inverse problem of recovering a symmetric free-energy potential and its interfacial scale from resolved phase-field observations. A physically interpretable reconstruction must first establish which combination of these quantities the observations actually determine.

The forward model has an exact scale ambiguity. For known diffusion coefficient $q$, the evolution depends on $F$ and $\eps$ only through the effective force
$G=F'/\eps^2$. The transformation
$(\eps,F)\mapsto(s\eps,s^2F)$, with $s>0$, leaves $G$ and hence every solution of the governing equation unchanged. More accurate or more densely sampled trajectories cannot distinguish members of this scale class, and neither can the corresponding equilibrium interface profile. An independent energy datum is therefore needed to recover an absolute scale. We use the planar surface tension in the normalization
$\sigma(F)=\int_{-1}^{1}\sqrt{2F(u)}\dd u$.
Once $G$ has been identified and the pure-phase energy has been set to zero, this datum determines $\eps$ and $F$ algebraically.

Constitutive inference from evolving patterns has been approached through optimization constrained by partial differential equations (PDEs) \cite{Zhao2020Pattern,Zhao2021ImageInversion}, variational system identification \cite{Wang2019VSI,Wang2021VSI}, and sparse regression \cite{Brunton2016SINDy,Rudy2017PDE}. Zhao, Braatz and Bazant \cite{Zhao2021ImageInversion} study how imaging conditions and physical constraints affect the recovery and uncertainty of constitutive functions. Weak formulations improve robustness to measurement noise by transferring derivatives from the observations to test functions \cite{Messenger2021WSINDy,Stephany2024WeakPDELearn}. Neural approaches include physics-informed neural networks \cite{Raissi2019PINN}, pseudo-spectral physics-informed neural networks (SPINN) \cite{Zhao2020SPINN}, and phase-field weak-form neural networks (PFWNN) \cite{Bao2025PFWNN}. Free-energy representations and differentiable evolution models offer further routes to constitutive inference \cite{Teichert2019IDNN,Cohen2026Differentiable}. These methods provide different estimators and representations, but no method using only the same trajectory data can recover a scale that is absent from the forward dynamics.

Scaling invariances also arise in broader parameter-identification problems \cite{Yates2009SymmetryIdentifiability,KianLiimatainenLin2024Gauge}. For the Cahn--Hilliard equation, Brunk, Egger and Habrich \cite{BrunkEggerHabrich2023} analyze uniqueness up to scaling and stable estimation by a regularized equation-error method. Uniqueness and numerical reconstruction results for phase-field potentials are available under other observation settings \cite{NiLai2024Uniqueness,NiLai2026DTPFI}. A second issue is structural: an estimated force may fit the data while producing the wrong number or type of stationary points. Structural information can be introduced through prior-informed weak dictionaries \cite{TangLiuWang2026PriorIDENT} or structure-preserving neural dynamical models \cite{LeeTraskStinis2022}, and interfacial energies have been used to determine parameters in prescribed phase-field potentials \cite{Balashov2026Tuning,TenEikelderBrunk2026Calibration,Wang2026Ferroelectric}. The present problem requires both a shape-preserving constitutive estimate and a separate choice of physical scale.

We address these requirements with surface-tension-calibrated Allen--Cahn identification (STAC). STAC separates the inverse problem into two operations. First, space--time weak moments yield a linear regression for the effective force without differentiating the measured field. Positive coefficients in a Bernstein representation impose the prescribed symmetric double-well structure through a convex least-squares problem. Second, one independently supplied surface tension selects the physical member of the scale class. The surface tension is an input to this calibration; it is not inferred from the trajectory used to estimate $G$.

The main contributions are:
\begin{enumerate}
\item an exact characterization of the scale equivalence in the Allen--Cahn forward model and a proof that one positive surface-tension datum uniquely selects $(\eps,F)$ within that class;
\item a constrained weak estimator whose accepted solutions retain two stable pure phases and a single central energy maximum;
\item finite calibration bounds for simultaneous perturbations in the recovered primitive and surface tension, including the physically relevant case in which the potential vanishes at the pure phases; and
\item controlled comparisons on identical weak systems that separate structural positivity from ridge regularization and approximation bias. Across 80 shared noisy records, positivity with ridge regularization removes six degree-five structural violations and lowers the mean force error from 8.70\% to 5.66\%. The comparisons also identify cases in which degree two is more accurate and an outside-family law for which the constraint adds bias.
\end{enumerate}

STAC requires resolved phase values, known diffusion, and a surface tension expressed in consistent energy units; interface contours alone are insufficient. The numerical study uses synthetic observations and treats constitutive degree as a prescribed modeling choice because the tested selection procedures were not reliable across all laws and noise levels. Section~\ref{sec:math_background} develops the identifiability and calibration results, Section~\ref{sec:methodology} presents the estimator, and Section~\ref{sec:component_results} analyzes recovery, structural failures, and sensitivity to the observations and supplied datum. Section~\ref{sec:conclusions} summarizes the findings and limitations. Material with an \textsf{S} prefix belongs to the Supplementary Material.

\section{Identifiability and surface-tension calibration}
\label{sec:math_background}
\subsection{Energy, mobility, and the identifiable force}
We consider the Allen--Cahn equation associated with the nondimensional energy
\begin{equation}
 E_\eps[u]=\int_\Omega\left(\frac{\eps}{2}|\nabla
 u|^2+\frac{F(u)}{\eps}\right)\dd\boldsymbol{x},
\label{eq:energy}
\end{equation}
where $\Omega\subset\mathbb R^2$, $\eps>0$, and the generating potential is smooth and even. Its
pure phases satisfy
\[
 F(-1)=F(1)=0,\qquad F'(-1)=F'(1)=0,\qquad
 F(u)>0\quad(-1<u<1).
\]
We fix the pure-phase energy at zero because the evolution equation contains only the derivative
of $F$.

Under periodic or homogeneous Neumann boundary conditions, the model is
\begin{equation*}
 u_t=-\frac{q}{\eps}\frac{\delta E_\eps}{\delta u}
 =q\left(\Delta u-\frac{F'(u)}{\eps^2}\right),\qquad q>0,
\end{equation*}
with a prescribed initial field. Here, $q$ is the diffusion coefficient and $q/\eps$ is the
$L^2$ mobility for $E_\eps$. For smooth solutions, $\dd E_\eps/\dd t=-(\eps/q)\int_\Omega
|u_t|^2\dd\boldsymbol{x}\leq0$. We set $q=1$ in the numerical experiments. The Allen--Cahn
equation does not conserve the spatial integral of $u$. An equivalent energy normalization is
\[
\mathcal E_\eps[u]=\eps E_\eps[u]
 =\int_\Omega\left(\frac{\eps^2}{2}|\nabla u|^2+F(u)\right)\dd\boldsymbol{x}
\]
with mobility $q/\eps^2$. The interfacial-energy datum must use the same normalization as the
calibration formula. Define
\[
G(u)=\frac{F'(u)}{\eps^2},\qquad
 H(u)=\int_{-1}^{u}G(z)\dd z=\frac{F(u)}{\eps^2}.
\]
\begin{proposition}[Scale equivalence and structural non-identifiability]
\label{thm:scale_equivalence}
Fix $q>0$ and let $I=[-1,1]$. For $\eps>0$ and $F\in C^1(I)$ with $F(-1)=0$, define
$\mathcal N(u;\eps,F)=u_t-q(\Delta u-F'(u)/\eps^2)$.
For every $s>0$ and every sufficiently regular $I$-valued field, $\mathcal
N(u;s\eps,s^2F)=\mathcal N(u;\eps,F)$.
Moreover, two normalized pairs have the same effective force on all of $I$ if and only if they
belong to the same equivalence class $[(\eps,F)]=\{(s\eps,s^2F):s>0\}$. Observations governed by
this operator therefore cannot distinguish members of a class under the same initial and
boundary data.
\end{proposition}
\begin{proof}
The ratio $(s^2F)'/(s\eps)^2=F'/\eps^2$ proves the invariance. Conversely, if
$F_1'/\eps_1^2=F_2'/\eps_2^2$ on $I$, integration from $-1$ and the normalization $F_i(-1)=0$
give $F_2=(\eps_2/\eps_1)^2F_1$; set $s=\eps_2/\eps_1>0$. Identical operators have identical
solution sets.
\end{proof}
The converse concerns equality of forces throughout $I$, not equality of one observed
trajectory. A single record can fail to distinguish additional forces because of limited phase
coverage or poorly conditioned moments, so the proposition establishes an unavoidable scale
ambiguity rather than complete identifiability modulo that ambiguity. The energy normalization
is essential: without $F(-1)=0$, an additive constant is also invisible to the dynamics.
Identification below concerns $G$ in a prescribed space; $H=F/\eps^2$ and the equilibrium
profile are invariant within its scale class.

\subsection{Calibration from interfacial energy}
\label{sec:sigma_theory}
For a monotone planar equilibrium connecting the two stable phases, equipartition gives
$\tfrac{\eps}{2}(u')^2=F(u)/\eps$, so that $|u'|=\sqrt{2F(u)}/\eps$ and the interfacial energy
associated with Eq.~\eqref{eq:energy} is
\[
 \sigma=\int_{-1}^{1}\sqrt{2F(z)}\dd z
       =\eps\int_{-1}^{1}\sqrt{2H(z)}\dd z,
\]
the classical relation for the stated normalization \cite{Cahn1958,Modica1987}. For $\mathcal
E_\eps$ the corresponding interfacial energy is instead $\eps\sigma$.

\begin{proposition}[Unique surface-tension selection within a scale class]
\label{thm:unique_calibration}
Let $(\eps,F)$ satisfy the normalization of Proposition~\ref{thm:scale_equivalence}, with
$F\geq0$ and $0<\sigma(F)<\infty$. For any supplied $\sigma_{\rm obs}>0$, exactly one member of
its scale class has that tension, namely $s_*=\sigma_{\rm obs}/\sigma(F)$ and
$(\widehat\eps,\widehat F)=(s_*\eps,s_*^2F)$. Equivalently, given a nonnegative normalized
primitive $\widehat H$ with $0<C_{\widehat H}=\int_I\sqrt{2\widehat H}\dd u<\infty$, the unique
calibrated pair with that primitive and tension is
\begin{equation}
 \widehat\eps=\frac{\sigma_{\rm obs}}{C_{\widehat H}},
 \qquad \widehat F(u)=\widehat\eps^{\,2}\widehat H(u).
 \label{eq:epsilon_from_surface_tension}
\end{equation}
\end{proposition}
\begin{proof}
Positive homogeneity gives $\sigma(s^2F)=s\sigma(F)$, whose strict monotonicity proves existence
and uniqueness of $s_*$. For a fixed primitive every pair in its class has the form
$(a,a^2\widehat H)$ with tension $aC_{\widehat H}$; solving $aC_{\widehat H}=\sigma_{\rm obs}$
yields Eq.~\eqref{eq:epsilon_from_surface_tension}.
\end{proof}
For a recovered admissible primitive, the supplied surface tension fixes the remaining scale
uniquely. Calibration leaves the effective force unchanged, so trajectory agreement tests the
recovered dynamics without independently verifying the supplied energy scale. Uniqueness is
conditional on the primitive and does not establish its recovery from one trajectory.

The interface width cannot replace this datum. The width between two interior phase levels
$-1<a<b<1$ is $W_{a,b}=\int_a^b(2H)^{-1/2}\dd u$, which is unchanged by the scale
transformation; a formula $W_{a,b}=C\eps$ can determine $\eps$ only after the amplitude and
shape of $F$ have been fixed. For curved, evolving interfaces the planar formula is approximate,
but the invariance of the full equation remains exact.

\subsection{Finite calibration stability and endpoint sensitivity}
The calibration map involves a square root that degenerates at the pure phases, so a bound valid
for potentials vanishing at $\pm1$ is needed. Theorem~\ref{thm:finite_calibration} provides one
without a Taylor expansion and without requiring $\widetilde H/H$ to stay bounded near the
endpoints.
\begin{theorem}[Finite perturbations of the calibration map]
\label{thm:finite_calibration}
Let $H,\widetilde H\in L^\infty(I)$ be nonnegative almost everywhere, where $I=[-1,1]$, and set
$C_H=\int_I\sqrt{2H}\dd u$. Then $|C_{\widetilde H}-C_H|\leq 2\sqrt{2\|\widetilde
H-H\|_\infty}$.
Suppose also $C_H,C_{\widetilde H}>0$ and $\sigma,\widetilde\sigma>0$. Define $\eps=\sigma/C_H$,
$F=\eps^2H$ and their tilded counterparts. The exact relative identity is
\begin{equation}
 r_\eps:=\frac{\widetilde\eps}{\eps}
 =\frac{\widetilde\sigma}{\sigma}\frac{C_H}{C_{\widetilde H}}.
 \label{eq:exact_calibration_ratio}
\end{equation}
Write $D=\|\widetilde H-H\|_\infty$, $\rho=2\sqrt{2D}/C_H$, and
$\widetilde\sigma=(1+\delta)\sigma$, $\delta>-1$. If $\rho<1$, then
\begin{equation}
 \frac{1+\delta}{1+\rho}\leq r_\eps\leq\frac{1+\delta}{1-\rho},
 \qquad |r_\eps-1|\leq\frac{|\delta|+\rho}{1-\rho}.
 \label{eq:finite_scale_bound}
\end{equation}
For all such calibrated pairs, regardless of whether $\rho<1$,
\begin{equation}
 \|\widetilde F-F\|_\infty
 \leq\eps^2\left[r_\eps^2D+|r_\eps^2-1|\,\|H\|_\infty\right].
 \label{eq:finite_potential_bound}
\end{equation}
\end{theorem}
\begin{proof}
For $a,b\geq0$, $|\sqrt a-\sqrt b|\leq\sqrt{|a-b|}$, hence $|C_{\widetilde
H}-C_H|\leq\sqrt2\int_I|\widetilde H-H|^{1/2}\dd u\leq2\sqrt{2D}$. Equation~\eqref{eq:exact_calibration_ratio} follows directly from calibration. With $a=(C_{\widetilde H}-C_H)/C_H$, the first bound gives
$|a|\leq\rho$, and since $r_\eps=(1+\delta)/(1+a)$ both inequalities in
Eq.~\eqref{eq:finite_scale_bound} follow. Finally $\widetilde F-F=\eps^2[r_\eps^2(\widetilde
H-H)+(r_\eps^2-1)H]$, and the triangle inequality proves Eq.~\eqref{eq:finite_potential_bound}.
\end{proof}
The interval in Eq.~\eqref{eq:finite_scale_bound} follows by allowing the relative integral change $a$ to range over $[-\rho,\rho]$. Its endpoints need not be attainable by a given pair of admissible primitives: the preceding square-root estimate can be conservative. The theorem is not needed when only the datum varies, because $\widehat\eps$ is
then exactly proportional to $\sigma_{\rm obs}$; it is needed when the primitive varies as well,
and stability of the inverse scale additionally requires $C_H$ to be bounded away from zero.

Two consequences are used below. Within the positive finite family of
Section~\ref{sec:constitutive_family}, suppose the reference coefficients satisfy $g_{*,j}>0$
for every $j$, with $H_*=\sum_jg_{*,j}P_j>0$ on $(-1,1)$; let $\eta=\max_j|\widehat
g_j-g_{*,j}|/g_{*,j}$ and $\sigma_{\rm obs}=(1+\delta)\sigma_*$ with $\delta>-1$, and assume
$\eta<1$. Then
\begin{equation}
 \frac{1+\delta}{\sqrt{1+\eta}}\leq\frac{\widehat\eps}{\eps_*}\leq
 \frac{1+\delta}{\sqrt{1-\eta}},\qquad
 (1+\delta)^2\frac{1-\eta}{1+\eta}\leq\frac{\widehat F(u)}{F_*(u)}\leq
 (1+\delta)^2\frac{1+\eta}{1-\eta}
 \label{eq:multiplicative_envelope}
\end{equation}
on $(-1,1)$, and along endpoint-controlled perturbations the first-order variation is
\begin{equation}
 \frac{\delta\eps}{\eps}
 \simeq\frac{\delta\sigma}{\sigma}-\frac{\delta C_H}{C_H},
 \qquad
 \delta C_H=\int_{-1}^{1}\frac{\delta H(z)}{\sqrt{2H(z)}}\dd z.
 \label{eq:sigma_sensitivity}
\end{equation}
Both statements, with their hypotheses and proofs, are Corollaries~S1 and~S2 of the
Supplementary Material. They are deterministic envelopes, and coefficients close to zero can
make $\eta<1$ fail or leave the bounds uninformative. For fixed $H$, changing $\sigma$ to
$(1+\delta)\sigma$ multiplies $\eps$ by $1+\delta$ and $F$ by $(1+\delta)^2$, whereas replacing
$\widehat H$ by $\alpha H$ with an exact tension gives $\widehat\eps=\eps/\sqrt\alpha$ and
$\widehat F=F$. The errors in $G$, $F$, and $\eps$ must therefore be distinguished;
Eq.~\eqref{eq:new_errors_v10} defines the numerical metrics, and
Section~\ref{sec:calibration_results} tests Eq.~\eqref{eq:sigma_sensitivity} against an
independently computed interfacial energy. The datum must be supplied in the normalization of
Eq.~\eqref{eq:energy}. If the datum $\gamma_{\rm obs}$ instead refers to $\mathcal E_\eps$, the consistent relation is $\gamma_{\rm obs}=\widehat\eps^2 C_{\widehat H}$ and hence $\widehat\eps=\sqrt{\gamma_{\rm obs}/C_{\widehat H}}$. Inserting that datum into Eq.~\eqref{eq:epsilon_from_surface_tension} without changing the normalization gives an incorrect scale even if the trajectory agrees. Physical units and the dissipation identity of the calibrated
model are given in Section~S2.

\section{Weak identification and constrained estimation}
\label{sec:methodology}
The estimator has two stages. The effective force is first determined from the observed phase
fields by a constrained weak solve. The supplied surface tension then fixes the potential and
the interfacial scale algebraically. Algorithm~\ref{alg:method} summarizes the sequence.

\subsection{A positive constitutive family}\label{sec:constitutive_family}
For $j=0,\ldots,m$, define $B_j^{(m)}(u)=-u(1-u^2)\binom{m}{j}(u^2)^j(1-u^2)^{m-j}$ on $[-1,1]$
and $P_j^{(m)}(u)=\int_{-1}^u B_j^{(m)}(z)\dd z$. Here $m$ is the degree of the Bernstein
bracket, not the degree of the full force or potential. The bracket is the Bernstein basis of
degree $m$ in $u^2$, so it is nonnegative on the phase interval, forms a partition of unity
there, and admits exact degree elevation \cite{Farin2002}; the elevation operator is what lets
one fixed design serve every candidate degree in Section~\ref{sec:component_effects}. These
primitives satisfy $(P_j^{(m)})'=B_j^{(m)}$, $P_j^{(m)}(\pm1)=0$, and $P_j^{(m)}\geq0$ on
$[-1,1]$. Thus $G_{\boldsymbol g}=\sum_{j=0}^m g_jB_j^{(m)}$ and $H_{\boldsymbol g}=\sum_{j=0}^m
g_jP_j^{(m)}$; at a fixed degree we suppress the superscript. Fixed degrees two and five carry
the principal comparisons; the exploratory observation-based selection of
Section~\ref{sec:component_effects} considers degrees one through five, and degrees zero through
five appear in the approximation study of Section~\ref{sec:replication}.

Let $v=1-u^2$. At Bernstein degree two the representation is
\begin{equation}
 G_{\boldsymbol g}(u)=-u(1-u^2)\big[g_0(1-u^2)^2+2g_1u^2(1-u^2)+g_2u^4\big],
 \qquad g_j>0,
 \label{eq:positive_force}
\end{equation}
for $|u|\leq1$, with the exact primitive in Eq.~\eqref{eq:positive_potential},
\begin{equation}
 H_{\boldsymbol g}(u)=v^2\left[\frac{g_0v^2}{8}+g_1\left(\frac v3-\frac{v^2}{4}\right)
 +g_2\left(\frac14-\frac v3+\frac{v^2}{8}\right)\right].
 \label{eq:positive_potential}
\end{equation}
The force in Eq.~\eqref{eq:positive_force} has exactly one zero in $(-1,1)$, at $u=0$; its
primitive satisfies $H(\pm1)=H'(\pm1)=0$, $H>0$ for $-1<u<1$, $H''(\pm1)=2g_2>0$, and
$H''(0)=-g_0<0$. Hence the potential has two stable pure phases and one interior stationary
point, the central maximum. The same conclusions hold at every degree $m$ whenever all $g_j>0$;
the estimator imposes $\boldsymbol g\geq\ell\boldsymbol1$ with $\ell>0$ in Eq.~\eqref{eq:positive_ls}. We distinguish the nonnegative cone $C_0$ from the numerical feasible set $C_\ell=\{\boldsymbol g:g_j\geq\ell\}$, a translated orthant. The small positive lower bound ensures nondegenerate well curvatures. On the closed cone $\{\boldsymbol g\geq\boldsymbol0,\ \boldsymbol
g\neq\boldsymbol0\}$ the force still has its only interior zero at $u=0$ and the primitive is
positive on $(-1,1)$, because the bracket is positive for $0<u^2<1$ and every $P_j^{(m)}$ is
positive there; the lower bound $\ell>0$ additionally makes the curvatures $H''(0)<0$ and
$H''(\pm1)>0$ strict. This representation is a sufficient but restrictive condition for a
symmetric double well; coefficient positivity is not necessary for physical admissibility. For
example, with $z=u^2$,
\begin{equation}
 p(z)=(z-\tfrac12)^2+0.01,\qquad G_{\rm c}(u)=-u(1-u^2)p(u^2).
 \label{eq:counterexample}
\end{equation}
Since $p>0$, its zero-endpoint primitive has stable pure phases and one central maximum. Its
degree-two Bernstein coefficients are $(0.26,-0.24,0.26)$, and exact elevation to degree five
gives $(0.26,0.06,-0.04,-0.04,0.06,0.26)$. The force belongs to both unconstrained spaces but to
neither nonnegative cone, and multiplication by any positive reaction-amplitude factor preserves
this exclusion and its relative approximation error.

Figure~\ref{fig:v12_approximation} shows the resulting approximation error, which is independent
of observation noise and therefore quantifies the cost of the structural guarantee before any
data are involved. We call the best nonnegative-cone relative force error the
\emph{approximation floor}. For this admissible law it is 37.68\% at degree two and 5.55\% at
degree five, whereas the unconstrained space of the same degree represents the law exactly. At
degree two the cone cannot reproduce the interior oscillation of $G_{\rm c}$ near $u=\pm0.8$; at
degree five it nearly can. These function-space projections are computed with 80-point
Gauss--Legendre quadrature and checked at 160 points. They are an evaluation-only lower bound
for a strictly positive fit, which separates approximation capacity from inference; they do not
test recovery from a trajectory of this law, which Section~\ref{sec:cone_trajectory} addresses.

Two distinct assumptions are therefore involved. The first is that the law has the prescribed
symmetric double-well structure. The second is that the law is well represented inside the
nonnegative cone at the declared degree. Equation~\eqref{eq:counterexample} satisfies the first
and fails the second. The structural assumption does not guarantee an accurate approximation at a fixed degree. The reference-based floor is unavailable in an inverse application, so family adequacy must instead be assessed through physical information and sensitivity to the observations and the chosen representation. The trajectory tests in Section~\ref{sec:cone_trajectory} show that a small approximation floor alone does not guarantee accurate recovery.

\begin{figure}[tbp]\centering
\includegraphics[width=\linewidth]{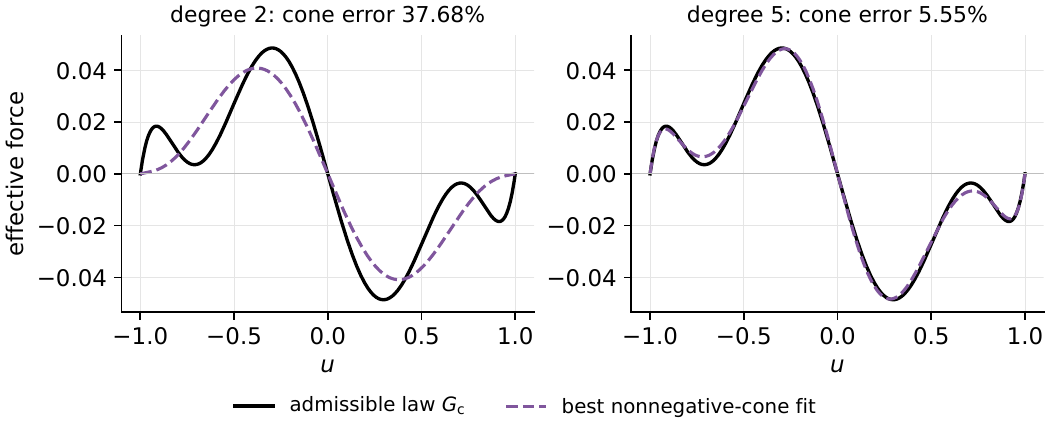}
\caption{Approximation limit of the nonnegative cone for the admissible law $G_{\rm c}$ of
Eq.~\eqref{eq:counterexample}. Left: Bernstein degree two; right: degree five. Solid line:
$G_{\rm c}$; dashed line: its best approximation with nonnegative Bernstein coefficients in
$L^2(-1,1)$. The panel titles give the relative $L^2$ force error of that approximation (the
approximation floor); the unconstrained space of each degree represents $G_{\rm c}$ exactly. No
trajectory data are involved.}\label{fig:v12_approximation}
\end{figure}

Outside the phase interval, with $x=\max(-1,\min(1,u))$, we define $H(u)=H(x)+g_m(u-x)^2$ and
$G(u)=G(x)+2g_m(u-x)$. This extension is twice continuously differentiable and its derivative is
the extended force; it is used when noisy observations take values outside $[-1,1]$.

\subsection{Space--time moments and their discrete evaluation}
\begin{lemma}[Space--time distributional identity]
\label{lem:weak_identity}
Let $\mathcal Q_T=\Omega\times(0,T)$ and $q>0$, and suppose $u,G(u)\in L^2(\mathcal Q_T)$. Then
$q^{-1}\partial_tu-\Delta u+G(u)=0$ in distributions if and only if, for every $\zeta\in
C_c^\infty(\mathcal Q_T)$,
\[
 \int_{\mathcal Q_T}G(u)\zeta\dd\boldsymbol x\dd t
 =\int_{\mathcal Q_T}u\left(\Delta\zeta+q^{-1}\partial_t\zeta\right)\dd\boldsymbol x\dd t.
\]
The identity extends to the closure $X_0$ of $C_c^\infty(\mathcal Q_T)$ in the test norm
$\|\zeta\|_X^2=\|\zeta\|_{L^2}^2+\|\partial_t\zeta\|_{L^2}^2+\|\Delta\zeta\|_{L^2}^2$.
\end{lemma}
Section~S2 proves the lemma. It requires no classical derivative of the observed field and
asserts only the interior equation, not an initial or boundary trace condition. The actual test
functions have compact interior supports and enough regularity to belong to $X_0$: the
zero-extended polynomial $(1-\xi^2)^5$ is $C^4$. With $\zeta(\boldsymbol x,t)=\phi(\boldsymbol
x)\psi(t)$ and $G=\sum_jg_jB_j$, the lemma yields the separable moment identity
\begin{equation}
 \sum_j g_j\int B_j(u)\phi\psi\dd\boldsymbol x\dd t
 =\int u\left(\Delta\phi\,\psi+q^{-1}\phi\psi'\right)\dd\boldsymbol x\dd t,
 \label{eq:weak_identity}
\end{equation}
in which the compact supports remove boundary terms. The lemma concerns the continuum weak
operator; it does not make quadrature, finite-difference weights, or nonlinear noisy features
exact.

The test function is $(1-\xi^2)^5$ on $|\xi|\leq1$ and zero elsewhere, tensorized in space, with
two paired supports: spatial half-widths $1/16$ and $1/8$ of the unit domain against temporal
half-widths of five and eight observation intervals. Two choices matter for what follows.
Centers are retained only where the complete derivative support lies inside the observations and
the temporal average satisfies $|\bar u|<0.98$, which excludes the bulk phases where the
features carry no information; and each support block is scaled by $(kn_b)^{-1/2}$ so that the
two contribute equally to the objective regardless of how many centers each retains. The
derivative weights are fourth-order centered differences of the sampled test weights carried
with their full margins, which preserves discrete summation by parts and annihilates constant
fields to roundoff; sampling the analytic derivatives instead does not. Fast Fourier transforms
evaluate the convolutions. Section~S1 states the weights, the spatial multiplier, and the
two-cell margin.

Stacking the normalized moments gives a matrix $A$, a vector $d$, and the problem
\begin{equation}
 \widehat{\boldsymbol g}=\operatorname*{arg\,min}_{\boldsymbol g\geq\ell\boldsymbol 1}
 \frac12\|A\boldsymbol g-d\|_2^2.
\label{eq:positive_ls}
\end{equation}
For the experiments in Section~\ref{sec:component_results}, the design is formed explicitly on
the paired temporal/spatial half-supports $(5\Delta t,1/16)$ and $(8\Delta t,1/8)$ with equal
block weighting. Column-scaled QR and singular value decompositions \cite{GolubVanLoan2013} and
the active-set nonnegative least-squares algorithm \cite{LawsonHanson1974} solve
Eq.~\eqref{eq:positive_ls} without forming normal equations, which would square the condition
number of the design. The system is rejected when the scaled singular-value ratio is at or below
$10^{-10}$. The lower bound is $\ell=10^{-10}\|\boldsymbol g_{\rm unc}\|_\infty$, and a zero
primitive is rejected at calibration. We record conditioning, active constraints, and optimality
residuals. Short and sparse records use the supports specified in Section~\ref{sec:stress}.
Section~S1 describes the supplementary normal-equation implementation.

The systems are large in rows and small in columns: on the principal $96^2$ records the two
supports retain $1.54\times10^{5}$ to $4.24\times10^{5}$ moment rows against $p=m+1\leq6$
columns, and QR compaction reduces each support to $p+1$ rows before merging. Assembly costs
$O(p\,n_tN^2\log N)$ and dominates the compaction and the solve; all quantities are evaluated in
double precision.

For the 80 principal high-noise records, column-scaled degree-five condition numbers range from 371 to 470. The corresponding noise-perturbation ratios, evaluated against clean common-mask estimates, range from 23.8 to 43.0. These values make worst-case coefficient bounds uninformative for those records; they do not measure the realized force error. The principal QR/SVD implementation avoids the additional loss of accuracy associated with forming normal equations.

\begin{algorithm}[tbp]
\caption{Surface-tension-calibrated Allen--Cahn identification}\label{alg:method}
\KwIn{Observed fields, mesh and times, known $q$, supplied $\sigma_{\rm obs}>0$, predeclared
degree and test supports.}
Use the predeclared positive constitutive basis and interior test supports\;
Assemble the weak moment system in Eq.~\eqref{eq:weak_identity} with complete derivative
margins\;
Check numerical rank and solve Eq.~\eqref{eq:positive_ls}, optionally with the penalty of
Eq.~\eqref{eq:ridge}; retain its conditioning and optimality diagnostics\;
Integrate the primitive and evaluate $C_{\widehat H}$ by 64-point Gauss--Legendre quadrature\;
Determine $\widehat\eps$ and $\widehat F$ from Eq.~\eqref{eq:epsilon_from_surface_tension}\;
\KwOut{Effective force, calibrated potential and scale, and numerical diagnostics.}
\end{algorithm}
Algorithm~\ref{alg:method} uses a predeclared degree and supports; degree two is the
low-complexity benchmark. Numerical rank, active constraints, and support sensitivity accompany
the estimate, with rank failure and inadmissible calibration rejected. Surface tension enters
only after force estimation.

\subsection{Constrained estimate, shrinkage, and conditional stability}
\label{sec:stability}
Equation~\eqref{eq:positive_ls} is convex, and its solution is the $Q$-metric projection of the
unconstrained least-squares solution onto the closed convex set $C_\ell$, where $Q=A^TA$ and $A$ has full column rank. It coincides with the
unconstrained solution exactly when that solution is already feasible, never increases the
observed moment residual relative to any feasible reference, and for a fixed design is Lipschitz
in the data with constant $1/s_{\min}(A)$. These are standard facts \cite{BoydVandenberghe2004},
stated and proved as Lemma~S1 of Section~S2 in the present notation; their role is to say when
the constraint changes an estimate at all. The contraction controls the observed moment norm,
not force error across the phase interval or forward field error, and the common-feasible-set assumption fails when $\ell$ is recomputed from each noisy record, even if the design is fixed.

The constraint and ordinary shrinkage act on different parts of the problem. Positivity
restricts the feasible coefficient directions without changing $Q$ or its eigenvalues, whereas
the ridge branch
\begin{equation}
 \min_{\boldsymbol g}\ \|A\boldsymbol g-d\|_2^2+\lambda\|\boldsymbol g\|_2^2,
 \qquad\text{with or without }\boldsymbol g\geq\ell\boldsymbol1,
 \label{eq:ridge}
\end{equation}
changes the Hessian, up to the common objective factor, to $Q+\lambda I$ while leaving the
feasible set untouched. This is ordinary ridge regression \cite{HoerlKennard1970}: it shifts
every eigenvalue of $Q$ by $\lambda$ and shrinks the estimate along weakly excited directions,
at the cost of regularization bias. With the normalized design used here, a nonvanishing penalty can retain this bias as more observations become available. Because the two mechanisms
are distinct, they can be combined, and Section~\ref{sec:component_effects} measures their
contributions separately and jointly on shared weak systems.

The unconstrained branch of Eq.~\eqref{eq:ridge} is a Tikhonov-regularized weak equation-error
estimator, the class analyzed for the Cahn--Hilliard equation in \cite{BrunkEggerHabrich2023}.
Tables~\ref{tab:v10_main} and~\ref{tab:v12_ridge} compare this estimator with the constrained branch on shared moments and equal tuning budgets. A change of basis $A_B=A_PT$ likewise does not guarantee
better conditioning: for full-rank matrices and nonsingular square $T$,
$\kappa_2(A_B)\leq\kappa_2(A_P)\kappa_2(T)$ is only an upper bound, so the Bernstein basis is
used to certify the declared structure rather than to promise spectral improvement. Section~S2
gives the associated optimality conditions.

Realized observation and discretization perturbations enter through a single residual. Fix a
candidate space and a common retained row set with identical normalization, suppose $\boldsymbol
g_*$ is feasible for the realized lower bound, and write $A_{\rm obs}=A_*+\Delta A$ and $d_{\rm
obs}=A_*\boldsymbol g_*+r_{\rm app}+r_{\rm disc}+\Delta d$. If $\|\Delta A\|_2<s_{\min}(A_*)$,
then
\begin{equation}
 \|\widehat{\boldsymbol g}-\boldsymbol g_*\|_2
 \leq\frac{\|r_{\rm app}\|_2+\|r_{\rm disc}\|_2+\|\Delta d\|_2+
 \|\Delta A\|_2\|\boldsymbol g_*\|_2}
 {s_{\min}(A_*)-\|\Delta A\|_2}=:R.
 \label{eq:observed_bound_v10}
\end{equation}
Equation~\eqref{eq:observed_bound_v10} also controls calibration.

Writing $\boldsymbol P(u)=(P_0(u),\ldots,P_{p-1}(u))^T$ and $M_P=\sup_{u\in I}\|\boldsymbol
P(u)\|_2$, Cauchy--Schwarz gives $\|\widehat H-H_*\|_\infty\leq M_PR$, so
Theorem~\ref{thm:finite_calibration} applies with $D\leq M_PR$; in particular
$2\sqrt{2M_PR}<C_{H_*}$ gives a finite scale-error certificate, and if
$g_{\min,*}=\min_jg_{*,j}>0$ with $R<g_{\min,*}$ then $\eta\leq R/g_{\min,*}<1$ in
Eq.~\eqref{eq:multiplicative_envelope}. These are Corollaries~S3 and~S4 of Section~S2. The
reference may be a feasible approximation when the generating force lies outside the chosen
family, in which case its mismatch contributes to $r_{\rm app}$; noise-induced changes to the
retained mask $|\bar u|<0.98$ fall outside the common-row assumption.

\subsection{Systematic and statistical residuals}
\label{sec:residuals}
These chains are reference-dependent and deterministic, and they do not follow from $\|u_{\rm
obs}-u\|_{L^2}$ alone. Two of the three residuals are systematic rather than random. The
approximation residual $r_{\rm app}$ is a model-discrepancy term in the sense of the calibration
literature, so absorbing it into an observation-noise model biases the identified coefficients
\cite{BrynjarsdottirOHagan2014}, and $r_{\rm disc}$ plays the same role for the forward operator
\cite{KaipioSomersalo2007}; Section~\ref{sec:constitutive_family} measures the corresponding force-approximation error for one admissible law outside the positive family. This function-space error is not the same as its weak-moment residual: a nonzero force discrepancy can be small or invisible in the retained moments. The nonpolynomial laws of Section~\ref{sec:protocol} also have nonzero function-space approximation error at every finite degree. The third is statistical and specific to
weak identification with nonlinear features: for a polynomial feature $B_j^{(m)}$ and centered
Gaussian noise of variance $\sigma_n^2$,
\begin{equation*}
 \mathbb E[B_j^{(m)}(u+\xi)]
 =B_j^{(m)}(u)+\frac{\sigma_n^2}{2}(B_j^{(m)})''(u)+O(\sigma_n^4),
\end{equation*}
so averaging reduces variance without removing the bias. Weak sparse identification of nonlinear dynamics (WSINDy) distinguishes noise robustness
from unconditional consistency \cite{MessengerBortz2022Consistency}, and extending the
errors-in-variables treatment of weak-form estimation of nonlinear dynamics (WENDy) \cite{BortzMessengerDukic2023WENDy} to these overlapping
PDE moments requires additional analysis; the present ensembles establish neither calibrated
coefficient intervals nor consistency at fixed nonzero noise.

For a known Gaussian noise variance, Hermite substitution gives an exact correction to the means of polynomial features \cite{NISTHermiteGenerating}. This statement requires polynomial evaluation on the whole real line, predetermined rows, and deterministic weights. It does not apply unchanged to the exterior continuation or the observation-dependent phase mask. Section~S4 tests unmasked polynomial features with an estimated variance; unbiased feature means do not imply unbiased coefficients or validation scores. Section~\ref{sec:complexity} discusses the resulting degree selection.

\subsection{Experimental variants and comparison design}
Constrained and unconstrained branches share the same degree, moments, and weights. The ridge
penalty is the identity in Bernstein coefficient coordinates with $\lambda=\alpha\|A_{\rm
train}^{(m)}\|_F^2/(m+1)$, $\alpha\in\{0,10^{-8},10^{-6},10^{-4},10^{-2},1\}$, fixed before the
comparison was run. 
Here $A_{\rm train}^{(m)}$ is the training design in the degree-$m$ Bernstein basis. The divisor
is $p=m+1$, hence three at degree two and six at degree five, so that $\lambda/\alpha$ is the
mean squared column norm at each degree. The selected numerical $\lambda$ is retained when
refitting the full record. Each branch selects only its own $\alpha$ by the minimum held-out
validation residual (ties to the smaller $\alpha$), so the two ridge branches have equal tuning
budgets (Section~S1). The analytic-derivative weak control uses positive coefficients and
selects degree and support from the same candidates and validation moments as the discrete-derivative
selection variant. The smoothed pointwise fourth-order finite-difference (FD4) control fixes
degree two and selects only the smoothing width. Section~S1 specifies these procedures.

\section{Numerical results}
\label{sec:component_results}
\subsection{Test problems and comparison protocol}
\label{sec:protocol}
The experiments examine whether the recovered law has the prescribed phase stability, how the structural constraint affects accuracy, and how the supplied tension changes the calibrated parameters. The principal benchmark uses circular and nonconvex initial fields with four symmetric potentials: the sixth-degree reference law, a degree-three Bernstein law with potential coefficients $(0.8,3.0,0.3,1.8)$, and
\[
 F_{\rm exp}(u)=\tfrac14(1-u^2)^2e^{u^2},\qquad
 F_{\rm rat}(u)=\frac{(1-u^2)^2}{4(1+0.7u^2)}.
\]
For the Bernstein law, $F=\sum_{j=0}^3b_jP_j^{(3)}$. The exponential and rational laws are outside every finite polynomial space, allowing approximation error to be distinguished from observation error. The initial field is held fixed across laws and is not assumed to be an equilibrium profile for each potential.

The unit-square domain has $q=1$, $\eps=0.03$, grid size $N=96$, and output interval $\Delta t=5\times10^{-5}$. An independent NumPy implementation of the second-order Crank--Nicolson--Adams--Bashforth method (CNAB2) generates 101 fields with 12 internal steps per output interval. The principal study contains eight clean law--geometry records and 80 records at each noise standard deviation, $0.01$ and $0.03$. Each noisy group uses ten paired Gaussian noise seeds per law--geometry pair. Degrees, supports, and selection rules were fixed for each comparison before its runs; the ridge comparison was added after the principal records had been examined and was subsequently repeated on fresh seeds.

Figure~\ref{fig:geometries} illustrates the three initial geometries on the separate $256^2$ sixth-degree-law dataset used for the neural comparison and supplementary controls. The inclusion retains an inner and an outer interface over the displayed interval. These images illustrate geometric variation; the four-law principal benchmark uses the $96^2$ protocol above. Additional wobble-shaped-circle and dumbbell tests are reported in Section~S3.

\begin{figure}[tbp]\centering
\includegraphics[width=0.94\linewidth]{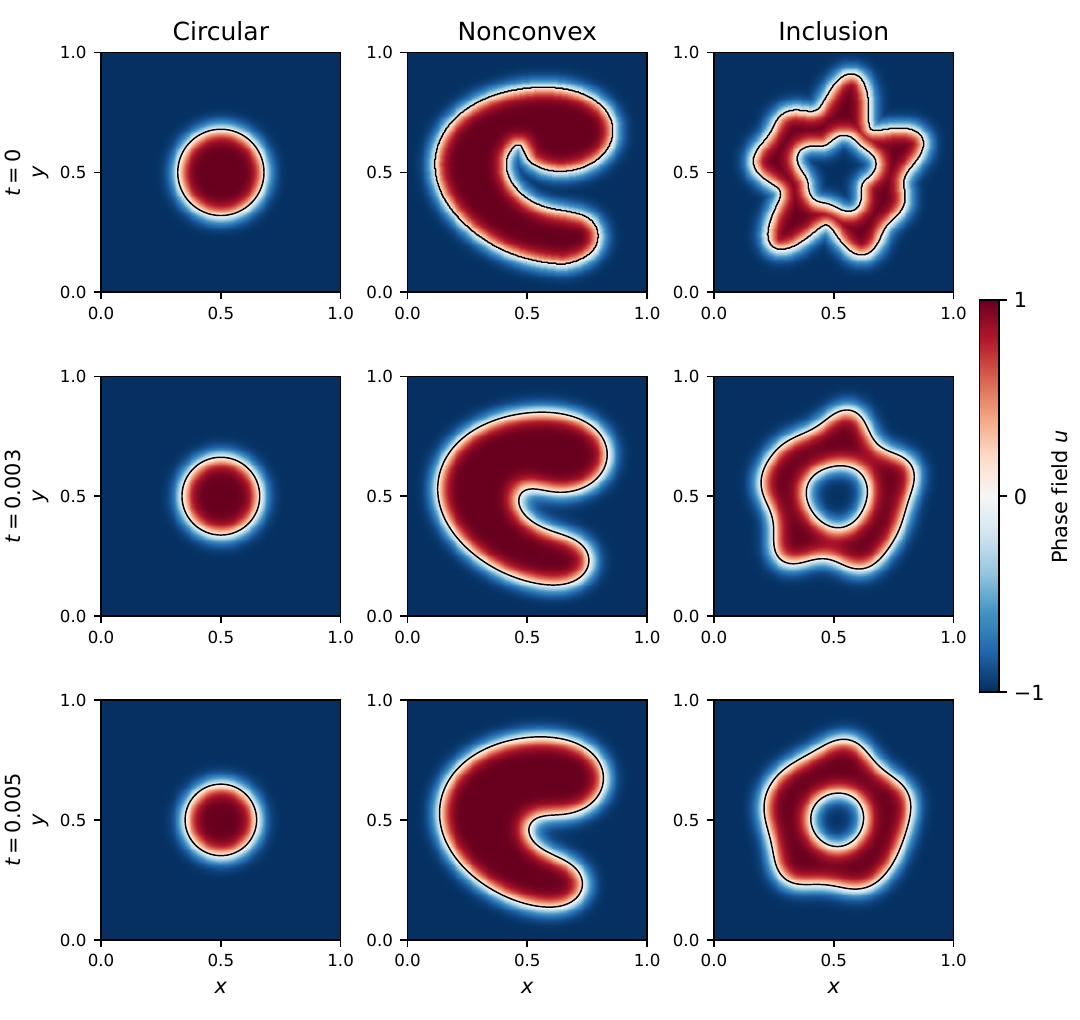}
\caption{Sixth-degree-law observations on the separate $256^2$ dataset: circular, nonconvex, and inclusion geometries at the initial time, identification endpoint, and final time. Colors show $u$ on a common scale and black contours mark $u=0$. Contours use the cell-centred coordinates of the fields. These are the neural-comparison observations, not the $96^2$ four-law principal ensemble.}\label{fig:geometries}
\end{figure}

All methods in a paired comparison use identical observations and noise realizations. Percentile bootstrap intervals \cite{EfronTibshirani1993} use 10,000 resamples of the ten seed clusters, retaining all law--geometry cases within each cluster. They summarize repeated noise on the tested cases. With only ten clusters, nominal 95\% coverage is not established; Section~S7 compares percentile, bias-corrected and accelerated, and cluster-$t$ intervals and identifies conclusions sensitive to the construction. The deterministic bounds of Section~\ref{sec:math_background} are separate from these statistical summaries.

The error measures are
\begin{equation}
\begin{aligned}
 &e_G=\frac{\|\widehat G-G_*\|_2}{\|G_*\|_2},\qquad
 e_F=\frac{\|\widehat F-F_*\|_2}{\|F_*\|_2},\qquad
 e_\eps=\frac{|\widehat\eps-\eps_*|}{\eps_*},\\
 &e_u=\frac{\|\widehat u(t_{100})-u_*(t_{100})\|_2}{\|u_*(t_{100})\|_2}.
\end{aligned}
\label{eq:new_errors_v10}
\end{equation}
Constitutive norms use 2001 uniformly spaced phase values; the separate $256^2$ comparisons use 20001. For $e_u$, both laws evolve from the same clean frame 61 to frame 100, at $t_{100}=0.005$, using a separate fourth-order exponential time-differencing Runge--Kutta solver (ETDRK4) \cite{KassamTrefethen2005}. No later observation enters this evolution. Thus $e_u$ tests the recovered force over this horizon, while $e_F$ and $e_\eps$ also depend on the supplied tension. Solver checks and failure handling are detailed in Sections~S1 and~S3. A structural violation denotes failure of the prescribed positive primitive or single interior stationary point; a negative Bernstein coefficient alone is not a violation. Failed calibration leaves $e_F$ and $e_\eps$ undefined, and failures remain in the reported counts.

\subsection{Effects of positivity and ridge regularization}
\label{sec:component_effects}
Table~\ref{tab:v10_main} separates positivity and ridge regularization on the same 80 principal records at noise $0.03$. At degree five, the mean force error decreases from 8.70\% for unconstrained least squares to 6.20\% with positivity, 7.29\% with ridge regularization, and 5.66\% with both. Positivity prevents coefficient combinations inconsistent with the prescribed wells, while ridge regularization reduces the response to weakly determined directions of the design.

\begin{table}[tbp]\centering\footnotesize
\setlength{\tabcolsep}{3.2pt}
\caption{Principal comparison at noise $0.03$ (four laws, two geometries; 80 records per method). All errors are means in percent. ``Principal seeds'' use the ten shared seeds of the benchmark; ``fresh seeds'' repeat the comparison with ten noise seeds never used in any earlier analysis (--- where not run). Positivity is never active at degree two on the principal seeds, so the constrained and unconstrained degree-two branches coincide and are reported as one row. The analytic weak control and smoothed FD4 are positive degree-two fits. ``Active'' counts constrained fits with at least one coefficient at the lower bound (--- for branches without a bound); ``Viol.'' counts recovered potentials that fail the prescribed double well. Every fit attains full numerical rank, calibrates, and completes its forward evolution.}\label{tab:v10_main}
\begin{tabular}{lrrrrrrrr}
\toprule
& \multicolumn{6}{c}{Principal seeds} & \multicolumn{2}{c}{Fresh seeds}\\
\cmidrule(lr){2-7}\cmidrule(lr){8-9}
Method & $e_G$ & $e_F$ & $e_\eps$ & $e_u$ & Active & Viol. & $e_G$ & Viol.\\\midrule
STAC $=$ weak LS, $m=2$ & 5.65 & 1.51 & 0.545 & 0.125 & 0/80 & 0/80 & 5.65 & 0/80\\
Ridge STAC $=$ ridge LS, $m=2$ & 5.75 & 1.53 & 0.533 & 0.123 & 0/80 & 0/80 & --- & ---\\
\midrule
Weak LS, $m=5$ & 8.70 & 1.63 & 0.479 & 0.111 & --- & 6/80 & 9.40 & 3/80\\
STAC, $m=5$ & 6.20 & 1.28 & 0.494 & 0.113 & 52/80 & 0/80 & 7.85 & 0/80\\
Ridge LS, $m=5$ & 7.29 & 1.46 & 0.475 & 0.112 & --- & 4/80 & 7.19 & 2/80\\
Ridge STAC, $m=5$ & 5.66 & 1.22 & 0.484 & 0.112 & 25/80 & 0/80 & 6.67 & 0/80\\
\midrule
Analytic weak & 6.69 & 1.89 & 0.251 & 0.107 & 0/80 & 0/80 & --- & ---\\
Smoothed FD4, $m=2$ & 13.1 & 2.19 & 5.48 & 0.995 & 0/80 & 0/80 & --- & ---\\
\bottomrule\end{tabular}\end{table}

The structural counts make this distinction clearer than the mean errors alone. Unconstrained degree-five regression violates the well structure in 6 of 80 records; ridge regularization reduces this count to four but does not eliminate it. Both constrained branches retain the prescribed structure in every record. The constraint is active in 52 unpenalized estimates and 25 ridge estimates. Regularization reduces the frequency with which the estimate reaches the coefficient bound, but does not replace the structural condition.

\begin{figure}[tbp]\centering
\includegraphics[width=\linewidth]{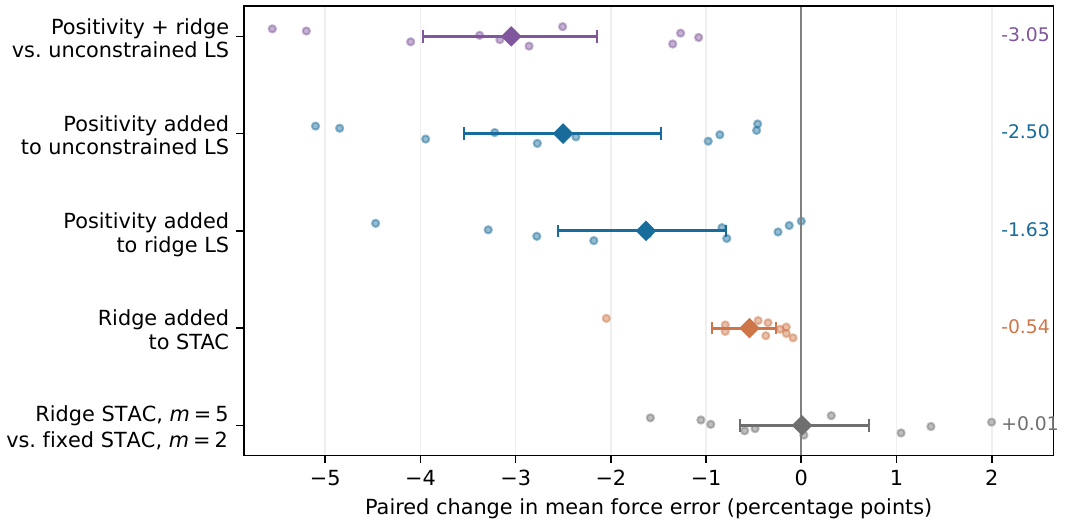}
\caption{Paired changes in mean force error at degree five on the principal records, in percentage points. Negative values favor the first-named method. Diamonds indicate means over ten seed clusters, dots show the individual cluster means, and bars show nominal 95\% percentile bootstrap intervals. The final comparison uses fixed degree-two STAC as the reference.}\label{fig:principal_ridge}
\end{figure}

The paired contrasts in Fig.~\ref{fig:principal_ridge} support both contributions on these records: adding positivity to ridge changes the mean error by $-1.63$ percentage points (pp), with interval $[-2.56,-0.79]$, and adding ridge to positivity changes it by $-0.54$~pp, $[-0.94,-0.26]$. The combined degree-five estimate is close to fixed degree two, with difference $+0.01$~pp, $[-0.64,+0.72]$. This agreement does not persist on the fresh seeds: ridge STAC at degree five gives 6.67\%, compared with 5.65\% at degree two. The fresh-seed mean improvement from adding positivity to ridge remains $-0.52$~pp, but its cluster-$t$ interval includes zero. The observed ordering is therefore more stable than a claim of statistical significance for every contrast.

Figure~\ref{fig:structure_example} illustrates the physical consequence of one structural violation. The unconstrained estimate introduces zeros near $u=\pm0.0835$ and changes the central energy maximum into a local minimum. The mixed state becomes locally metastable rather than unstable, although the overall force curve is close to the reference. This is an error in local phase stability that a small average trajectory error need not reveal. The estimate constrained on the same moments retains the required double well. The example is illustrative; its selection is specified in Section~S1.

\begin{figure}[tbp]\centering
\includegraphics[width=\linewidth]{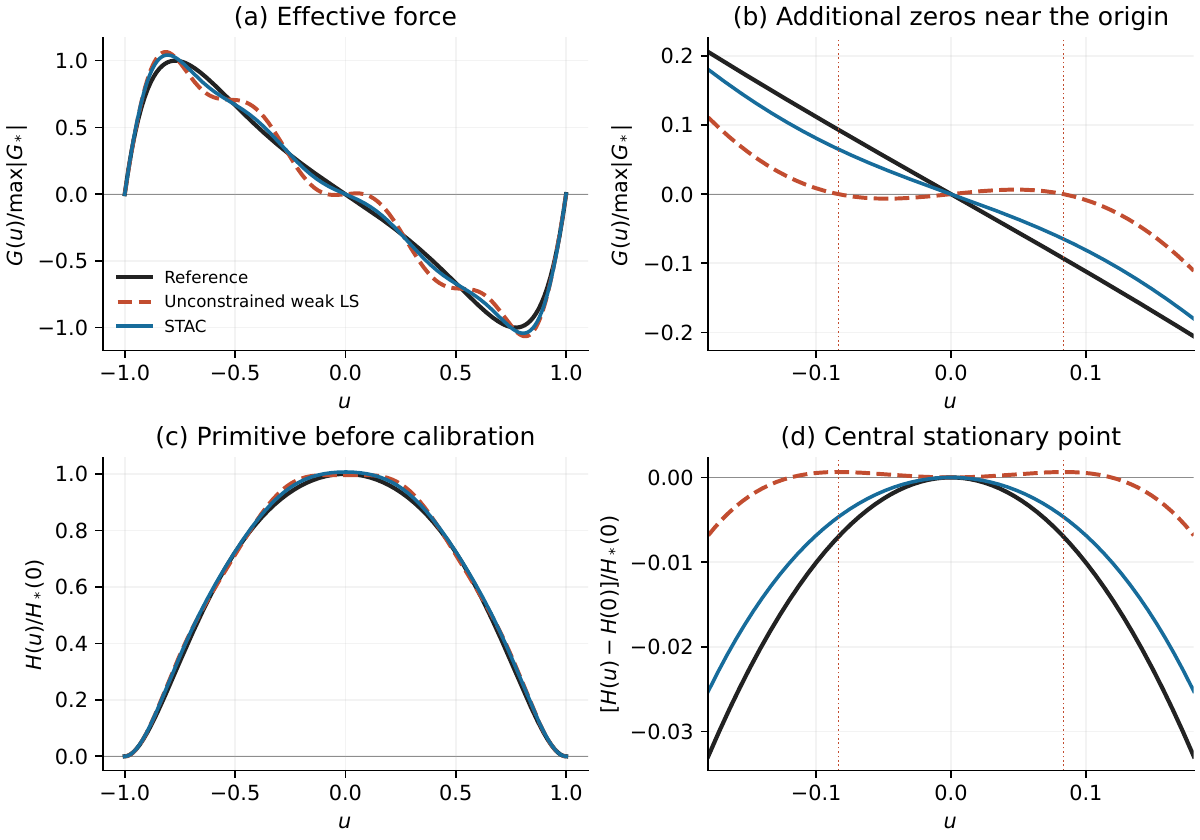}
\caption{Degree-five recovery of the exponential law from circular observations at noise $0.03$. The unconstrained estimate (dashed) and STAC (blue) use the same moments; black is the reference. Panels (a,b) show the force and an enlargement near the origin. Panels (c,d) show the primitive and its local change from $u=0$. The additional force zeros and central local minimum identify an unintended metastable mixed state.}\label{fig:structure_example}
\end{figure}

The law-by-law results explain why the two degrees can have similar pooled errors while behaving differently (Table~\ref{tab:v19_law}). Degree two has a 6.45\% approximation floor for the degree-three law, and degree-five ridge STAC reduces its observed error by 7.09~pp. For the other three laws, the degree-two floor is at most 0.60\%; additional coefficients increase the error by 1.13--3.40~pp. These changes are consistent with a bias--variance tradeoff: a richer family is useful when its reduction in approximation error exceeds the increase in sensitivity to the observations.

\begin{table}[tbp]\centering\small
\caption{Law-by-law force error on the 80 principal records at noise $0.03$ (20 records per law: two geometries, ten seeds). All entries are relative force errors in percent; $\Delta$ is Ridge STAC $m=5$ minus STAC $m=2$ in percentage points. The cone floor is the data-free best nonnegative-cone approximation of the reference force (0.00 denotes below 0.005\%).}\label{tab:v19_law}
\setlength{\tabcolsep}{5pt}\begin{tabular}{lrrrrr}\toprule Law & \shortstack{Cone floor\\$m=2$} & \shortstack{Cone floor\\$m=5$} & \shortstack{STAC\\$m=2$} & \shortstack{Ridge STAC\\$m=5$} & $\Delta$ (pp)\\\midrule
Sixth-degree & 0.00 & 0.00 & 1.25 & 3.84 & $+2.59$\\
Bernstein degree 3 & 6.45 & 0.00 & 15.90 & 8.81 & $-7.09$\\
Exponential & 0.60 & 0.00 & 4.19 & 7.59 & $+3.40$\\
Rational & 0.44 & 0.00 & 1.26 & 2.39 & $+1.13$\\
\bottomrule\end{tabular}\end{table}

At degree two, positivity is inactive on the principal records and ridge regularization changes little. Consequently, the improvement over smoothed fourth-order finite-difference regression (FD4), $-7.43$~pp with interval $[-7.52,-7.34]$, cannot be attributed to positivity; both methods use the same positive family. It concerns the weak and pointwise formulations under the tested smoothing and noise settings. The ordering reverses on the separate clean $256^2$ controls, where direct FD4 is more accurate (Section~S3). Neither result supports a universal ranking.

Finally, $e_u$ is only 0.111--0.125\% across the fixed weak branches despite their larger differences in $e_G$. The common short validation horizon and phase values visited by the trajectories make the final field less discriminating than the full constitutive curve. Agreement in $e_u$ therefore complements force error and local-stability checks.

\subsection{Replication and the effect of constitutive degree}
\label{sec:replication}
A second ensemble uses four laws, one geometry, and fresh seeds for full and short records at noise $0.01$ and $0.03$, giving 40 records per condition. In the comparison with observation-selected configurations, constrained and unconstrained estimates use the same selected moments. The constrained mean error is smaller in all four conditions, while unconstrained structural violations increase from 1 to 12 of 40 as the observations become less informative (Section~S4). Some interval conclusions depend on the construction; the short-record contrast at noise $0.01$, for example, has a cluster-$t$ interval containing zero.

\begin{table}[tbp]\centering\small
\caption{Constrained and ridge-regularized fits on the matched replication ensemble (four laws, one geometry, ten fresh seeds; 40 records per cell). Entries are mean relative force errors $e_G$ in percent, with the number of structure violations among the 40 fits in parentheses. The first data column is the \emph{constrained} STAC estimator at degree two; the remaining four columns are degree-five branches on the same moments.}\label{tab:v12_ridge}
\setlength{\tabcolsep}{4pt}\begin{tabular}{llrrrrr}\toprule
& & \multicolumn{1}{c}{$m=2$} & \multicolumn{4}{c}{$m=5$}\\
\cmidrule(lr){3-3}\cmidrule(lr){4-7}
Record & Noise & STAC & LS & STAC & Ridge LS & Ridge STAC\\\midrule
full & 0.01 & 4.83 (0) & 7.56 (2) & 6.13 (0) & 4.37 (0) & 3.83 (0)\\
full & 0.03 & 5.67 (0) & 11.67 (3) & 9.21 (0) & 10.62 (2) & 8.84 (0)\\
short & 0.01 & 4.66 (0) & 12.59 (8) & 8.85 (0) & 5.90 (2) & 5.02 (0)\\
short & 0.03 & 9.50 (0) & 25.03 (11) & 15.55 (0) & 21.85 (5) & 15.16 (0)\\
\bottomrule\end{tabular}\end{table}

Table~\ref{tab:v12_ridge} fixes degree five and adds ridge regularization on these same records. The combined estimator has the smallest mean error among the degree-five branches in each condition, and only the constrained branches avoid all structural violations. Its comparison with degree two is different: at noise $0.03$, degree-two STAC is more accurate for both record lengths. Replication thus supports the separate effects of positivity and regularization without establishing that the richer family is preferable.

The clean circular tests in Fig.~\ref{fig:degree} provide a second explanation for this behavior. The best approximation of the exponential force improves with degree, but recovery from the trajectory is most accurate at degree four and deteriorates at degree five. The increased condition number exposes coefficient combinations that the record determines poorly. These supplementary calculations use the earlier normal-equation implementation, so their high-degree behavior can also contain numerical error from forming the Gram matrix; they should not be read as a conditioning test of the principal QR/SVD solver.

\begin{figure}[tbp]\centering
\includegraphics[width=\linewidth]{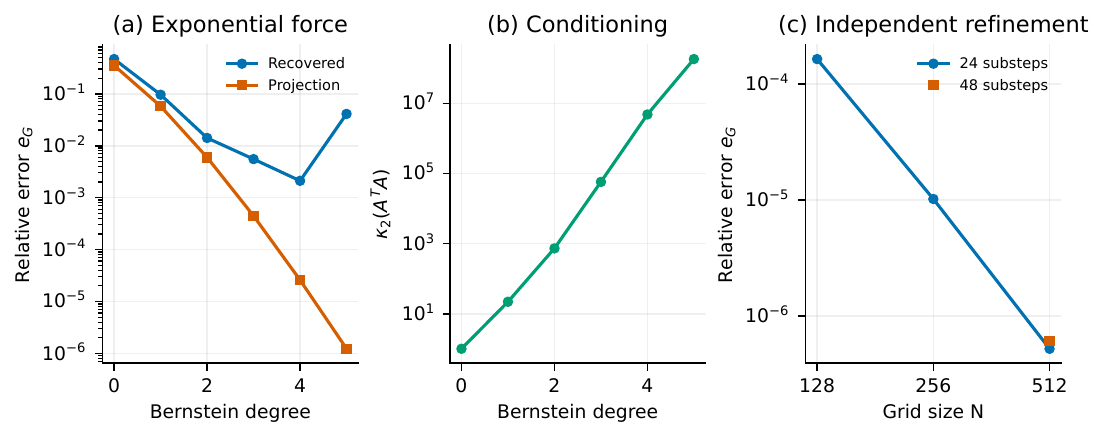}
\caption{Supplementary clean-data tests. (a) Exponential-force recovery and best nonnegative approximation against Bernstein degree. (b) Condition number of the weak Gram matrix in the earlier normal-equation implementation used for these records. (c) Sixth-degree-law recovery under spatial refinement with a common initial field and 24 or 48 CNAB2 substeps. These calculations are separate from the principal QR/SVD comparisons.}\label{fig:degree}
\end{figure}

For the sixth-degree law, spatial refinement reduces the clean force error from 0.0163\% at $N=128$ to $5.3\times10^{-5}$\% at $N=512$. This indicates a small discretization contribution in that resolved, clean test. It does not by itself separate discretization and observation errors for every law or establish consistency at fixed nonzero noise.

\subsection{Recovery when the true law is outside the positive family}
\label{sec:cone_trajectory}
The nonpolynomial benchmark laws have small approximation errors at the tested degrees. The law in Eq.~\eqref{eq:counterexample} provides a more demanding test: it has the required symmetric double well and belongs to the unconstrained polynomial space, but its Bernstein coefficients lie outside the nonnegative cone at degrees two and five. Its amplitude is chosen to match the resolved interface width of the reference law. This comparison tests a restriction on an otherwise adequate candidate space.

\begin{table}[tbp]\centering\footnotesize
\setlength{\tabcolsep}{3pt}
\caption{Trajectory recovery of the admissible law of Eq.~\eqref{eq:counterexample}, which lies
outside the nonnegative cone at both degrees. Relative force error is the mean and sample standard deviation over both geometries: two deterministic estimates for clean data and 20 estimates from two geometries and ten paired noise seeds otherwise; both branches
use identical weak moments. ``Floor'' is the data-free cone approximation limit of
Fig.~\ref{fig:v12_approximation}, and the next column is the constrained error as a multiple of it.
The last column counts how often unconstrained least squares already lands inside the cone.}
\label{tab:cone_trajectory}
\begin{tabular}{@{}rrrrrrrrrr@{}}\toprule
$m$ & noise & unconstrained & constrained & cost & floor & con./floor
& \multicolumn{2}{c}{admissible} & uncon.\\
\cmidrule(lr){8-9}
 & & $e_G$ (\%) & $e_G$ (\%) & [pp] & (\%) & & uncon. & con. & in cone\\\midrule
2 & 0.00 & $0.45\pm0.02$ & $60.82\pm0.12$ & $+60.38$ & 37.68 & 1.61 & 2/2 & 2/2 & 0/2\\
2 & 0.01 & $2.64\pm0.19$ & $61.31\pm0.14$ & $+58.67$ & 37.68 & 1.63 & 20/20 & 20/20 & 0/20\\
2 & 0.03 & $77.80\pm0.82$ & $77.80\pm0.82$ & $+0.00$ & 37.68 & 2.06 & 20/20 & 20/20 & 20/20\\
5 & 0.00 & $1.57\pm0.29$ & $6.01\pm0.01$ & $+4.44$ & 5.55 & 1.08 & 2/2 & 2/2 & 0/2\\
5 & 0.01 & $13.15\pm2.05$ & $6.93\pm0.07$ & $-6.22$ & 5.55 & 1.25 & 4/20 & 20/20 & 0/20\\
5 & 0.03 & $10.00\pm2.88$ & $27.89\pm0.23$ & $+17.89$ & 5.55 & 5.03 & 20/20 & 20/20 & 0/20\\
\bottomrule\end{tabular}\end{table}

At degree two, imposing positivity increases the clean force error from 0.45\% to 60.82\%, and the error at noise $0.01$ from 2.64\% to 61.31\%. The unconstrained estimates are already structurally admissible in these 22 records. The constraint introduces a large bias without a structural benefit. Moreover, the constrained errors exceed the 37.68\% function-space approximation floor: weak moments weight the phase values visited by the trajectory differently from the uniform phase-space norm used to define that floor.

At degree five and noise $0.01$, the balance reverses. The constrained error is 6.93\%, compared with 13.15\% without the constraint, and 16 of 20 unconstrained estimates violate the prescribed structure. Four constrained coefficients reach the lower bound in every record, suppressing large oscillations in the unconstrained coefficients. The approximation floor is now only 5.55\%, so the restriction can reduce noise sensitivity enough to compensate for excluding the exact law.

This benefit is not determined by the approximation floor alone. At noise $0.03$, with the same degree-five floor, the constrained error rises to 27.89\% and exceeds the unconstrained error of 10.00\%. Nonlinear feature bias and changes in the observation-dependent mask alter the weak system as the noise increases. At degree two, both estimates coincide at this noise level because the unconstrained coefficients are already positive, despite a force error of 77.80\%. An inactive constraint therefore provides no evidence that the candidate family accurately describes the underlying law.

These results separate physical admissibility from approximation quality. Positivity guarantees the declared well structure, but its usefulness for recovery depends jointly on the degree, trajectory, and observation error. A reference-based approximation floor explains the synthetic tests; it is unavailable as a selection criterion when the constitutive law is unknown.

\subsection{Limits of observation-based degree selection}
\label{sec:complexity}
The pooled principal and fresh-seed results favor a low degree, whereas the degree-three law benefits from a richer family. The choice is therefore not resolved by a pooled accuracy ranking. It requires a criterion that recognizes approximation error from the observations without giving excessive weight to noise-sensitive coefficients.

The exploratory selection rule does not achieve this consistently. Selecting degree and support from disjoint validation blocks improves the mean force error at noise $0.01$ from 4.81\% to 2.90\%, but increases it at noise $0.03$ from 5.65\% to 6.78\%. For the degree-three law, it selects degree one in all 20 high-noise records, although the law-by-law comparison shows a benefit from additional coefficients. A small validation residual is thus insufficient to identify the most accurate constitutive representation under these conditions.

Correcting polynomial feature means does not resolve the problem. On predetermined unmasked rows, the Hermite comparison of Section~S4 reduces the degree-two mean force error from 10.72\% to 8.19\%, but changes the selected-degree error by only $-0.05$~pp, with interval $[-1.30,+1.28]$. This is an estimated-variance implementation of an exact identity for known Gaussian variance; the identity does not imply unbiased coefficients or validation scores. The limited improvement in selection suggests that feature-mean bias is only one source of error.

The present implementation therefore prescribes the degree. This is a practical conclusion about the tested criteria and observation regime, not a proof that constitutive complexity cannot be inferred from data. Numerical rank, support sensitivity, and phase coverage remain useful diagnostics, but reliable adaptive degree selection requires further study.

\subsection{Effect of observation quality}
\label{sec:stress}
Six additional conditions change the information available for estimation: short records, sparse sampling in time, a $48^2$ grid, spatially correlated noise, and interfaces with $\eps=0.02$ or $0.05$. Each uses two laws and ten paired seeds at noise $0.03$, for 120 records in total. Supports are shortened when necessary, and fitting and validation frames remain disjoint. Section~S1 specifies the filtering, support sizes, and interface resolutions.

\begin{figure}[tbp]\centering
\includegraphics[width=\linewidth]{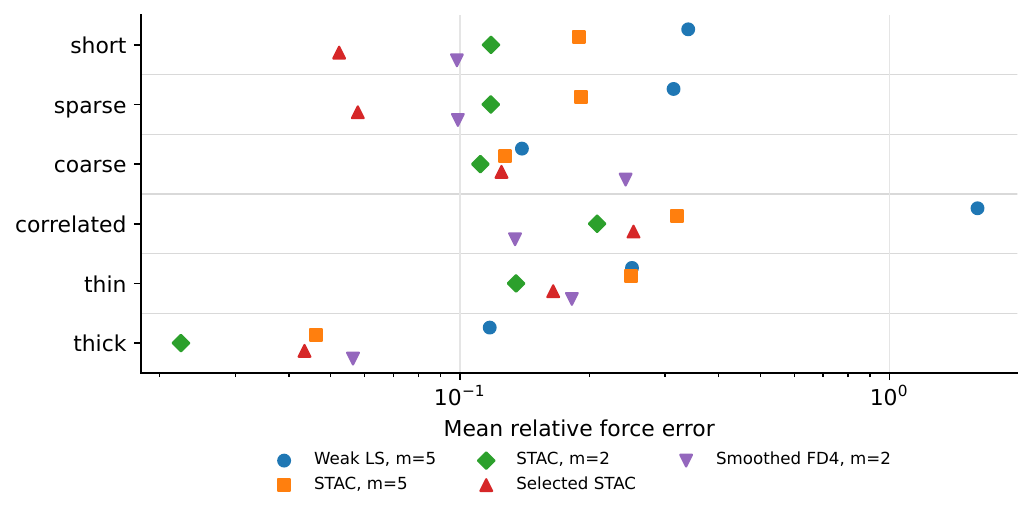}
\caption{Mean relative force error under six observation conditions, with two laws and ten seeds per condition. All force estimates are included, including structurally invalid unconstrained estimates. The selected STAC branch uses the exploratory observation-based degree and support rule.}\label{fig:v10_stress}
\end{figure}

Coarse observations and the thinner interface raise the fixed degree-two STAC error to 11.2\% and 13.5\%, respectively, compared with 2.7\% in the reference configuration. Both changes reduce the number of samples across the transition layer. The correlated-noise results in Fig.~\ref{fig:v10_stress} also show why averaging alone is insufficient: neighboring samples no longer provide independent fluctuations, and increasing the degree can amplify the remaining errors.

\begin{table}[tbp]\centering\small
\caption{Constraint activation and structural violations over the 120 observation-stress records. Activation is not applicable to unconstrained methods, denoted by a dash. The selected branches use the exploratory degree and support rule.}\label{tab:stress_counts}
\begin{tabular}{lrr}\toprule
Method & Active constraint & Structure violation\\\midrule
Weak LS, $m=2$ & --- & 41/120\\
STAC, $m=2$ & 41/120 & 0/120\\
Weak LS, $m=5$ & --- & 55/120\\
STAC, $m=5$ & 113/120 & 0/120\\
Selected weak LS & --- & 34/120\\
Selected STAC & 34/120 & 0/120\\
\bottomrule\end{tabular}\end{table}

The constraints preserve the prescribed structure even when recovery becomes inaccurate (Table~\ref{tab:stress_counts}). Unconstrained degrees two and five violate it in 41 and 55 of 120 records; degree five also fails calibration in three records. No constrained estimate violates the well structure. At still larger noise, $0.05$ and $0.10$, fixed-degree STAC force errors reach approximately 13\% and 25--33\%, while retaining admissibility (Section~S4). These are limitations in quantitative recovery, not evidence that the constraints recover information absent from the record.

Phase coverage helps explain where the errors occur. Weighting the force error by the observed phase histogram reduces the principal degree-five means from 8.70\% to 6.44\% without constraints and from 6.20\% to 4.77\% with constraints. The largest discrepancies are therefore concentrated partly in phase values sampled less often. Reporting only a trajectory-weighted error would conceal this limitation when the law is used with a different initial field.

\subsection{Comparison with neural constitutive estimates}
\label{sec:published}
Table~\ref{tab:published} compares STAC with the SPINN and PFWNN adaptations described in Section~S6. This is a separate $256^2$ sixth-degree-law study at noise $0.01$, with one measurement realization per geometry and three network initializations. The networks and STAC receive identical observations before conversion to each method's numerical precision. Terminal networks are used throughout.

\begin{table}[tbp]\centering\small
\setlength{\tabcolsep}{3pt}
\caption{STAC and adaptations of two published neural inverse methods on identical observations ($256^2$ records of the sixth-degree law, noise $0.01$, one measurement seed; three geometries). Entries are ranges over the three geometries; neural entries are means over three network initializations. Errors are relative and in percent. ``Raw'' is the network force as produced and ``projected'' its projection onto the positive family of Section~\ref{sec:constitutive_family}; potential errors are reported after projection. Details are given in Section~S30.}\label{tab:published}
\begin{tabular}{lcrrrr}
\toprule
Method & \shortstack{Admissible\\as produced} & \shortstack{$e_G$ (\%)\\raw} & \shortstack{$e_G$ (\%)\\projected} & \shortstack{$e_F$ (\%)\\projected} & Time\\\midrule
STAC & yes & \multicolumn{2}{c}{0.218--0.363} & 0.065--0.090 & 1.2--1.3~s CPU\\
SPINN \cite{Zhao2020SPINN} & no & 2.64--5.45 & 0.200--0.252 & 0.053--0.068 & \multirow{2}{*}{165--382~s GPU}\\
PFWNN \cite{Bao2025PFWNN} & no & 3.77--13.2 & 2.29--12.9 & 0.598--3.75 & \\
\bottomrule\end{tabular}\end{table}

The unconstrained neural forces do not satisfy the prescribed double-well conditions in these runs. Projection onto the common positive degree-two family both enforces those conditions and changes the constitutive error. For SPINN, projection reduces $e_G$ from 2.64--5.45\% to 0.200--0.252\% across geometries, making it comparable to, and in some cases more accurate than, STAC. This large change shows that postprocessing supplies substantial structural information; projected results should not be attributed to the network alone. PFWNN has larger errors under the reported training budget.

STAC requires approximately 1.2--1.3~s of CPU identification, compared with 165--382~s of GPU training for the neural methods. These measurements describe the present implementations and different hardware, rather than a general speedup factor. The comparison does not imply that every neural constitutive method needs this particular projection: admissibility could instead be imposed through its representation or objective. Three initializations on one noise realization quantify optimization variability, not performance over measurement noise.

\subsection{Interpretation of the calibrated potential and scale}
\label{sec:calibration_results}
Calibration can give different error levels for the potential and interfacial scale because the latter depends only on the integral $C_{\widehat H}$. Figure~\ref{fig:calibrated} shows this distinction for the four principal laws. Shape errors remain visible in $\widehat F$ even when their contribution to the calibration integral is small. The signed degree-two scale bias changes from $-0.0615\%$ at noise $0.01$ to $-0.5446\%$ at noise $0.03$, approximately the ninefold increase associated with a squared-noise term. This is consistent with a contribution from nonlinear feature bias, but does not isolate that mechanism. Signed bias, mean absolute error, and root-mean-square error are reported separately in Section~S5.

\begin{figure}[tbp]\centering
\includegraphics[width=\linewidth]{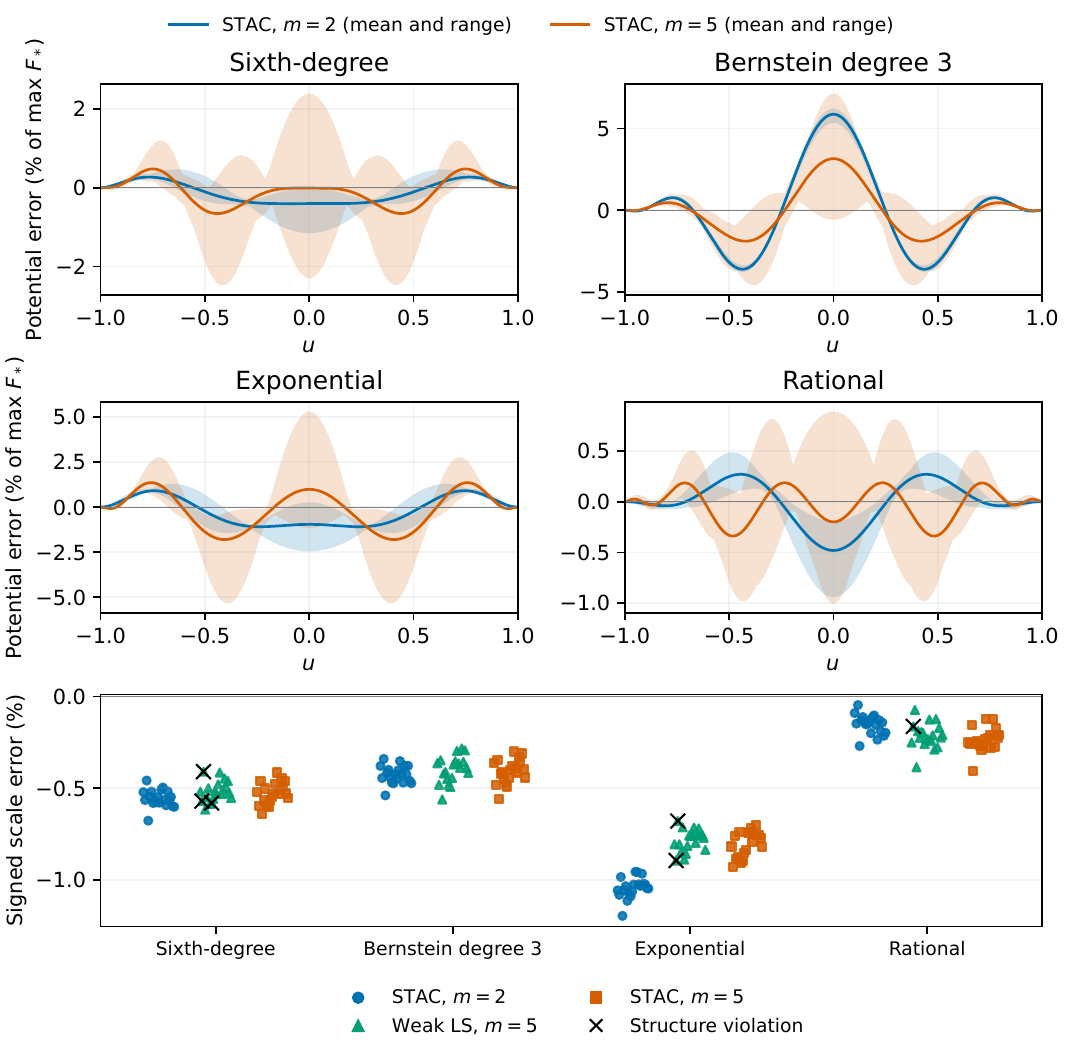}
\caption{Calibrated outputs on the principal records at noise $0.03$. First two rows: potential deviation in percent of $\max F_*$, with means and ranges over two geometries and ten seeds per law; vertical scales differ between panels. Bottom: signed relative scale error for each estimate. Crosses identify unconstrained estimates with structural violations. A small scale error can coexist with a larger error in potential shape.}\label{fig:calibrated}
\end{figure}

An independent one-dimensional equilibrium calculation checks the effect of numerically supplying the tension. Its finest mesh gives relative tension errors of $(1.7$--$2.6)\times10^{-4}\%$, with negligible changes in the reported calibrated errors. An under-resolved calculation gives tension errors of 0.61--0.93\% and increases the mean scale and potential errors to 1.24\% and 2.30\%. The significance is the relative size of the two error sources: an accurate equilibrium calculation contributes little, whereas an inaccurate datum can dominate an otherwise accurate force estimate. Since both calculations use the reference potential, they check numerical consistency rather than independent experimental calibration.

A separate sensitivity test gives each of the 80 field realizations at noise $0.01$ 200 independent, mean-one lognormal tension factors. It uses the exploratory observation-selected estimator and therefore has a different baseline from fixed-degree STAC. Increasing the tension coefficient of variation to 10\% raises the root-mean-square scale and potential errors from 0.0557\% and 1.45\% to 10.3\% and 21.1\%, respectively. The conditional tension draws do not increase the number of independent field realizations.

For fixed $\widehat H$, a factor $s$ in the supplied tension changes $\widehat\eps$ by $s$ and $\widehat F$ by $s^2$ exactly. The approximately 0.1\% crossover reported for this selected-estimator test is consequently specific to its identification error, not a required accuracy for all surface-tension measurements. Published interfacial energies can depend on thermodynamic state, orientation, and the model used to compute them \cite{DiPasquale2025}. Values from different calculations should not be interpreted as a measurement confidence interval. An experimental application must supply a physically matched datum in consistent units and propagate its uncertainty; no material-specific uncertainty is inferred from the present synthetic study.

\Needspace{30\baselineskip}
\section{Conclusions}
\label{sec:conclusions}
The central result is a separation of information in the Allen--Cahn inverse problem. Resolved trajectories determine the effective force $G=F'/\eps^2$, but they cannot determine $F$ and $\eps$ individually. STAC estimates $G$ from space--time weak moments and then uses an independently supplied surface tension to select the physical scale. The scale-equivalence result makes this limitation explicit, while the calibration theorem and perturbation bounds show when the selected pair is unique and how errors in both the recovered primitive and the supplied tension affect it.

The numerical comparisons clarify the roles of the two regularizing mechanisms. Bernstein positivity preserves the prescribed phase stability, whereas ridge regularization controls poorly determined coefficient directions. On the shared degree-five ensemble, their combination removes all six structural violations and reduces the mean force error from 8.70\% to 5.66\%. This improvement is not uniform: degree two is more accurate in several tests, and the outside-family example shows that a structural constraint can increase error through approximation bias. The constraint should therefore encode credible constitutive information rather than serve as a general guarantee of accuracy.

The present results are limited to resolved synthetic phase fields, known diffusion, and symmetric equal-depth wells. Applications to experimental data will also require uncertainty estimates for the measured surface tension, since agreement with a trajectory cannot validate the absolute energy scale. Reliable degree selection, confidence intervals, and extensions to incomplete observations and broader potential classes remain subjects for further study.

\section*{Acknowledgments}
This work was supported by the National Research Foundation of Korea(NRF) grant funded by the Korea government(MSIT) (No. RS-2026-25589376).

\begingroup
\setstretch{1}
\footnotesize
\setlength{\bibsep}{0pt plus 0.2ex}
\bibliographystyle{elsarticle-num}
\bibliography{references_v24}
\endgroup
\end{document}